\documentclass{amsart}
\usepackage{amsmath}
\usepackage[ocgcolorlinks,linktoc=all]{hyperref}
\hypersetup{citecolor=blue,linkcolor=red}
\newtheorem{theorem}{Theorem}

\newtheorem*{thmmain}{Theorem}
\newtheorem{lemma}[theorem]{Lemma}
\newtheorem{proposition}[theorem]{Proposition}

\theoremstyle{definition}
\newtheorem{definition}[theorem]{Definition}

\theoremstyle{remark}

\numberwithin{equation}{section}

\DeclareMathOperator{\RR}{\mathbb{R}}
\DeclareMathOperator{\sphere}{\mathbb{S}}
\DeclareMathOperator{\basepoint}{p_0}

\DeclareMathOperator{\radialdistance}{\rho}

\DeclareMathOperator{\reflectionvector}{V}
\newcommand{\reflectionplane}[1][\reflectionvector]{\ensuremath{P_{#1}}}
\newcommand{\reflectionmap}[1][\reflectionvector]{\ensuremath{R_{#1}}}
\newcommand{\reflectionset}[2][\reflectionvector]{\ensuremath{{#2}_{#1}}}
\newcommand{\reflectionhalfspace}[1][\reflectionvector]{\ensuremath{\reflectionset[{#1}]{H}}}
\newcommand{\vertvec}{e}
\DeclareMathOperator{\origin}{O}
\DeclareMathOperator{\radialprojection}{\pi}
\DeclareMathOperator{\height}{h}

\newcommand{\ip}[2]{\ensuremath{\langle{#1},{#2}\rangle}}
\DeclareMathOperator{\intersect}{\cap}

\begin{document}

\title[Harnack estimate for mean curvature flow on the sphere]
{Harnack estimate for mean curvature flow on the sphere}
\author[P. Bryan]{Paul Bryan}
\address{Department of Mathematics, Macquarie University
NSW 2109, Australia}
\email{paul.bryan@mq.edu.au}
\author[M.N. Ivaki]{Mohammad N. Ivaki}
\address{Department of Mathematics, University of Toronto, Ontario,
M5S 2E4, Canada}
\email{m.ivaki@utoronto.ca}

\begin{abstract}
We consider the evolution of hypersurfaces on the unit sphere $\mathbb{S}^{n+1}$ by their mean curvature. We prove a differential Harnack inequality for any weakly convex solution to the mean curvature flow. As an application, by applying an Aleksandrov reflection argument, we classify convex, ancient solutions of the mean curvature flow on the sphere.
\end{abstract}
\maketitle
We consider convex hypersurfaces $M_0$, $n\ge2$, without boundary, which are smoothly embedded in $\mathbb{S}^{n+1}.$ Let $F_0:M^n\to \mathbb{S}^{n+1}$ be a smooth embedding of $M_0$. The mean curvature flow with initial data $F_0$ is a family of hypersurfaces given by embeddings $F:M^n\times[0,T)\to \mathbb{S}^{n+1}$, $~ F(\cdot,0)=F_0(\cdot)$ that moves in the direction of the unit normal vector $\nu$ with speed equal to the mean curvature $H$, the trace of the second fundamental form $A(V,V)$ over the tangent vectors $V$, e.q.,
\[\partial_tF(\cdot,t)=-H(\cdot,t)\nu(\cdot,t).\]
Here, $H(p,t)$ is the mean curvature of the hypersurface $M_t:=F(M^n,t)$ at the point $F(p,t)$ where the unit outer normal vector is $\nu(p,t).$
Let $D$ denote the Levi-Civita connection of $\mathbb{S}^{n+1}.$
\begin{thmmain}[Harnack estimate]
Suppose $M_t$ is a smooth, weakly convex solution of the mean curvature flow on the time interval $[0,T)$. For $t>0$ we have
\[\partial_tH+\frac{H}{2t}+2D_VH+A(V,V)-nH\geq 0\]
for all tangent vectors $V$.
\end{thmmain}
Comparing this with Hamilton's Harnack inequality \cite{Hamilton 95} for the mean curvature flow in $\mathbb{R}^{n+1}$, here we gained a bonus term $-nH$ through the ambient space $\mathbb{S}^{n+1}.$ This allows us,
by employing a parabolic Aleksandrov reflection argument first developed in \cite{ Chow 97,Chow-Gul 01, Chow-Gul 96},  a rather convenient classification of ancient, smooth, weakly convex solutions of the mean curvature flow.
\begin{thmmain}[Classification of ancient solutions]
The only ancient, weakly convex, embedded solutions of the mean curvature flow on the sphere are equators and shrinking geodesic spheres.
\end{thmmain}
Let us note that such a classification has been obtained already in \cite[Theorem 6.1]{Hu-Sin 2014}, with a relatively short proof applying the maximum principle to the quantity $(\|A\|^2 - \tfrac{1}{n}H^2)/H^2$. As usual for such arguments, the Codazzi equation is employed; therefore, the result pertains to $n\geq 2$. The first author and Louie \cite{Br-Lou} used a completely different argument to obtain the classification for $n=1$. First, a Harnack inequality is obtained, which ensures that the backwards limit as $t \to -\infty$ is an equator. Then the Aleksandrov reflection argument is employed to show maximal reflection symmetry, immediately resulting in the classification. Our work here extends the techniques of \cite{Br-Lou} from $n=1$ to $n\geq 2$ formally following the same procedure of obtaining the Harnack inequality and then using the Aleksandrov reflection to show the maximal reflection symmetry. We believe, that unlike the methods in \cite{Hu-Sin 2014}, our argument is more broadly applicable: once the Harnack inequality is obtained for a curvature flow, the remainder of the argument holds with very little modification. This theme will be explored for a range of curvature flows similar to those studied in \cite{Andrews 94} in a forthcoming paper. In \cite{Andrews 94}, to obtain Harnack inequalities, the Gauss map parametrization and support function are used to great effect, considerably simplifying the computations. A genuine difficulty on the sphere is obtaining a similar useful parametrization under which we can perform the calculations.

Before moving on to the details, a few words about the broader context are in order. It is well known that ancient solutions arise as singularity models for curvature flows (e.g., \cite{Hu-Sin 1999}) and their study is of great importance. On the unit sphere, as opposed to the Euclidean case studied in \cite{Hu-Sin 2014}, where classification can be obtained only with additional assumptions such as pinching, we find a complete classification and strong rigidity: the only weakly convex, ancient solutions are the shrinking geodesic spheres. The difference is the compactness and positive curvature of the sphere.

Classifications of ancient solutions have been obtained for the curve shortening flow in the plane \cite{Da-Ham-Sesu 2010} and the Ricci flow of surfaces \cite{Da-Ham-Sesu 2012}. Both examples include self-similar shrinkers (round circles for the former and constant curvature spheres for the latter) and a unique Type II ancient solution with curvature blowing up faster than $|t|$ as $t\to-\infty$ (the Angenent oval [paper clip] in the former case and the Rosenau in the latter). In higher dimensions, classification results for ancient solutions in Euclidean space are less complete, and a wider range of possibilities may occur. As mentioned above, the results of \cite{Hu-Sin 2014} require additional pinching assumptions to deduce that such ancient solutions are shrinking spheres. Wang \cite{Wang} studied translating ancient solutions and discovered the existence of non-rotationally symmetric translators (though they blow down to spheres or cylinders). Furthermore, there are solutions asymptotic to ``ovals" in that the center looks like a cylinder, but the ends look like ``bowls" \cite{Ang,Has-Her,Whi}. These solutions have similarities to the Angenent oval in the plane.
\section*{Acknowledgement}
The first author would like to thank Bennett Chow for helpful discussions and encouragement. He would also like to thank the second author and the Institut f\"ur Diskrete Mathematik und Geometrie, TU Wien for hosting a very enjoyable and productive visit to Vienna. The work of the second author was supported by the Austrian Science Fund (FWF) Project M1716-N25.
\section{Preliminaries}
Let $g=\{g_{ij}\}$, $A=\{h_{ij}\},$ and $Rm_{ijkl}$ denote, in order, the induced metric, the second fundamental form, and the curvature tensor of $M^n$. The mean curvature of $M^n$ is the trace of the second fundamental form with respect to $g$, $H=g^{ij}h_{ij}.$ Let $\overline{g}$ and $\overline{Rm}_{\alpha\beta\gamma\theta}$ denote, respectively, the metric and the curvature tensor of $\mathbb{S}^{n+1}$. Greek indices run through $\{0,\ldots,n\}$ and Latin indices belong to the set $\{1,\ldots,n\}.$

Write $\nu$ for the outer unit normal to $M_t.$ For a fixed time, we choose a local orthonormal frame $\{\partial_0=\nu,\ldots,\partial_i=\frac{\partial F(\cdot,t)}{\partial x_i},\ldots,\partial_n\}$ in $\mathbb{S}^{n+1}.$ We use the following standard notation
\[h_i^j=g^{mj}h_{im}\]
\[(h^2)_i^j=g^{mj}g^{rs}h_{ir}h_{sm}\]
\[|A|^2=g^{ij}g^{kl}h_{ik}h_{lj}=h_{ij}h^{ij}\]
\[C=g^{ij}g^{kl}g^{mn}h_{ik}h_{lm}h_{nj}=h_i^kh_k^lh_l^i.\]
Here, $\{g^{ij}\}$ is the inverse matrix of $\{g_{ij}\}.$

The relations between $A$, $Rm$, and $\overline{Rm}$ are given by the Gau{\ss} and Codazzi equations:
\[Rm_{ijkl}=\overline{Rm}_{ijkl}+h_{ik}h_{jl}-h_{il}h_{jk}\]
\[\nabla_ih_{jk}=\nabla_{k}h_{ij}.\]
Moreover, $\nabla_i$ and $\Delta$ commute as follows
\[(\nabla_i\Delta-\Delta\nabla_i)f=-Rc_i^j\nabla_jf\]
for all smooth functions on $M^n.$
Therefore, in view of $\bar{R}_{\alpha\beta\gamma\theta}=\lambda (\overline{g}_{\alpha\gamma}\overline{g}_{\beta\theta}-\overline{g}_{\alpha\theta}\overline{g}_{\beta\gamma})$, $\lambda =1$, and the Gau{\ss} equation we have
\begin{equation}\label{eq: commute}
(\nabla_i\Delta-\Delta\nabla_i)H=((h^2)_i^m-Hh_i^m-(n-1)\delta_i^m\lambda )\nabla_mH.
\end{equation}
\section{Evolution equations}
In this section, we assume that $M_t$ is a strictly convex solution of the mean curvature flow.
\begin{lemma}\label{lem: lem3}
The following evolution equations hold under the mean curvature flow:
\begin{enumerate}
  \item $\partial_tg_{ij}=-2Hh_{ij}$
  \item $\partial_tg^{ij}=2Hg^{im}g^{jn}h_{ij}=2Hh^{ij}$
  \item $\partial_t h_i^j=\nabla^j\nabla_iH+H(h^2)_i^j+\lambda H\delta_i^j$
  \item $\partial_t h_i^j=\Delta h_i^j+|A|^2h_i^j+\lambda \{2H\delta_i^j-nh_i^j\}$
  \item $\partial_t h_{ij}=\Delta h_{ij}+|A|^2h_{ij}-2H(h^2)_{ij}+\lambda \{2Hg_{ij}-nh_{ij}\}$
  \item $\partial_t H=\Delta H+H|A|^2+n\lambda H$
  \item $(\partial_t\Delta-\Delta\partial_t)H=2Hh^{ij}\nabla_i\nabla_jH+2h^{ij}\nabla_iH\nabla_jH.$
\end{enumerate}
\end{lemma}
\begin{proof} The computations for $(1)-(6)$ are straightforward; see \cite{Huisken 87}. To obtain $(7)$, we note that
\begin{align*}
(\partial_t\Delta-\Delta\partial_t)H=\left(\partial_i\partial_jH-\Gamma_{ij}^k\partial_kH\right)\partial_tg^{ij}-g^{ij}\partial_kH\partial_t\Gamma_{ij}^k,
\end{align*}
where $\Gamma_{ij}^k=g^{kl}(\partial_ig_{jl}+\partial_jg_{il}-\partial_lg_{ij}).$ Since $\partial_t\Gamma_{ij}^k$ is tensorial, using $(1)$ and $ (2)$ we may carry out our calculations in a normal frame to prove the claim.
\end{proof}
\begin{lemma}\label{lem: lem1}
Under the mean curvature flow we have
\begin{align*}
\partial_t(\partial_tH)=&
\Delta \partial_tH+4Hh^{ij}\nabla_i\nabla_jH+2h^{ij}\nabla_iH\nabla_jH\\
&+(|A|^2+n\lambda )(\partial_t H)+2H^2C+2\lambda H^3,
\end{align*}
and
\begin{align*}
\nabla_i\partial_tH
=\Delta\nabla_iH+\nabla_i(|A|^2H)+((h^2)_i^m-Hh_i^m)\nabla_mH+\lambda \nabla_iH.
\end{align*}
\end{lemma}
\begin{proof} Using $(3), (6)$ and $(7)$ in Lemma \ref{lem: lem1}, we calculate
\begin{align*}
\partial_t(\partial_tH)=&\partial_t(\Delta H+H|A|^2+n\lambda H)\\
=&\Delta \partial_tH+2Hh^{ij}\nabla_i\nabla_jH+2h^{ij}\nabla_iH\nabla_jH\\
&+ (|A|^2+n\lambda )(\partial_t H)+H(\partial_t|A|^2)\\
=&\Delta \partial_tH+2Hh^{ij}\nabla_i\nabla_jH+2h^{ij}\nabla_iH\nabla_jH\\
&+ (|A|^2+n\lambda )(\partial_t H)+H(2h^{ij}\nabla_i\nabla_jH+2HC+2\lambda H^2)\\
=&\Delta \partial_tH+4Hh^{ij}\nabla_i\nabla_jH+2h^{ij}\nabla_iH\nabla_jH\\
&+(|A|^2+n\lambda )(\partial_t H)+2H^2C+2\lambda H^3.
\end{align*}
To obtain the second evolution equation, we use identity (\ref{eq: commute}) and part $(6)$ of Lemma \ref{lem: lem1}:
\begin{align*}
\nabla_i\partial_tH=&\nabla_i(\Delta H+|A|^2H+n\lambda H)\\
=&\Delta\nabla_iH+\nabla_i(|A|^2H)+n\lambda \nabla_iH\\
&+((h^2)_i^m-Hh_i^m-(n-1)\delta_i^m\lambda )\nabla_mH\\
=&\Delta\nabla_iH+\nabla_i(|A|^2H)+((h^2)_i^m-Hh_i^m)\nabla_mH+\lambda \nabla_iH.
\end{align*}
\end{proof}
In the sequel, $\{b^{ij}\}$ denotes the inverse of the second fundamental form.
\begin{lemma}\label{lem: lem 3}
\begin{align*}
\partial_t( b^{ij}\nabla_iH\nabla_jH)\leq&-b^{im}b^{jn}\Delta h_{mn}\nabla_iH\nabla_jH+2b^{ij}\nabla_jH\Delta\nabla_iH\\
&+2h^{ij}\nabla_iH\nabla_jH
+|A|^2b^{ij}\nabla_iH\nabla_jH+2Hb^{ij}\nabla_i|A|^2\nabla_jH\\
&+n\lambda b^{ij}\nabla_iH\nabla_jH.
\end{align*}
\end{lemma}
\begin{proof}
Observe
\begin{align*}
-\lambda b^{im}b^{jn}\{2Hg_{mn}-nh_{mn}\}\nabla_iH\nabla_jH+2\lambda b^{ij}\nabla_iH\nabla_jH\leq n\lambda b^{ij}\nabla_iH\nabla_jH
\end{align*}
and
\begin{align*}
\partial_t b^{ij}=&-b^{im}b^{jn}\partial_th_{mn}\\
=&-b^{im}b^{jn}(\Delta h_{mn}+|A|^2h_{mn}-2H(h^2)_{mn})\\
&-\lambda b^{im}b^{jn}\{2Hg_{mn}-nh_{mn}\}\\
=&-b^{im}b^{jn}\Delta h_{mn}-|A|^2b^{ij}+2Hg^{ij}\\
&-\lambda b^{im}b^{jn}\{2Hg_{mn}-nh_{mn}\}.
\end{align*}
Therefore, using the second part of Lemma \ref{lem: lem1}, we obtain
\begin{align*}
\partial_t( b^{ij}\nabla_iH\nabla_jH)=&-b^{im}b^{jn}(\Delta h_{mn}+|A|^2h_{mn}-2H(h^2)_{mn})\nabla_iH\nabla_jH\\
&-\lambda b^{im}b^{jn}\{2Hg_{mn}-nh_{mn}\}\nabla_iH\nabla_jH\\
&+2b^{ij}(\Delta\nabla_iH+\nabla_i(|A|^2H)+((h^2)_i^m-Hh_i^m)\nabla_mH+\lambda \nabla_iH)\nabla_jH\\
\leq& -b^{im}b^{jn}\Delta h_{mn}\nabla_iH\nabla_jH+2b^{ij}\nabla_jH\Delta\nabla_iH\\
&+2b^{im}b^{jn}H(h^2)_{mn}\nabla_iH\nabla_jH\\
&+|A|^2b^{ij}\nabla_iH\nabla_jH+2Hb^{ij}\nabla_i|A|^2\nabla_jH\\
&+2b^{ij}((h^2)_i^m-Hh_i^m)\nabla_mH\nabla_jH\\
&+n\lambda b^{ij}\nabla_iH\nabla_jH\\
=&-b^{im}b^{jn}\Delta h_{mn}\nabla_iH\nabla_jH+2b^{ij}\nabla_jH\Delta\nabla_iH\\
&+2h^{ij}\nabla_iH\nabla_jH
+|A|^2b^{ij}\nabla_iH\nabla_jH+2Hb^{ij}\nabla_i|A|^2\nabla_jH\\
&+n\lambda b^{ij}\nabla_iH\nabla_jH.
\end{align*}
\end{proof}
Using the identities
\[\nabla_mb^{ij}=-b^{ik}b^{jl}\nabla_mh_{kl}\]
and
\begin{align*}
\Delta b^{ij}=&-b^{im}b^{jn}\Delta h_{mn}\\
&+\{b^{ir}b^{ks}b^{jl}+b^{ik}b^{jr}b^{ls}\}g^{pq}\nabla_ph_{rs}\nabla_qh_{kl},
\end{align*}
we compute
\begin{align*}
\Delta( b^{ij}\nabla_iH\nabla_jH)=&-b^{im}b^{jn}\nabla_iH\nabla_jH\Delta h_{mn}+2b^{ij}\nabla_jH\Delta\nabla_iH\\\
&+2\{b^{ir}b^{ks}b^{jl}\}\nabla_iH\nabla_jHg^{pq}\nabla_ph_{rs}\nabla_qh_{kl}\\
&-4g^{mn}b^{ik}b^{jl}\nabla_mh_{kl}\nabla_n\nabla_iH\nabla_jH+2b^{ij}g^{mn}(\nabla_m\nabla_iH\nabla_n\nabla_jH).
\end{align*}
Thus we have proved:
\begin{lemma}\label{lem: lem2}
\begin{align*}
\partial_t(b^{ij}\nabla_iH\nabla_jH)\leq &\Delta (b^{ij}\nabla_iH\nabla_jH)\\
&-2\{b^{ir}b^{ks}b^{jl}\}\nabla_iH\nabla_jHg^{pq}\nabla_ph_{rs}\nabla_qh_{kl}\\
&+4g^{mn}b^{ik}b^{jl}\nabla_mh_{kl}\nabla_n\nabla_iH\nabla_jH-2b^{ij}g^{mn}(\nabla_m\nabla_iH\nabla_n\nabla_jH)\\
&+2h^{ij}\nabla_iH\nabla_jH+|A|^2b^{ij}\nabla_iH\nabla_jH+2Hb^{ij}\nabla_i|A|^2\nabla_jH\\
&+n\lambda b^{ij}\nabla_iH\nabla_jH.
\end{align*}
\end{lemma}
\begin{lemma}\label{lem: lema6}
Define $\Theta:=\partial_tH-b^{ij}\nabla_iH\nabla_jH$. Then
\begin{align*}
\partial_t\Theta&\geq\Delta \Theta+ \frac{2(\Theta-n\lambda H)^2}{H}+(|A|^2+n\lambda )\Theta\\
&+2\{g^{mq}b^{np}-\frac{g^{mn}g^{pq}}{H}\}\eta_{mn}\eta_{pq}+2\lambda H^3,
\end{align*}
where
\[\eta_{mn}=\nabla_m\nabla_nH+H(h^2)_{mn}-b^{rs}\nabla_rH\nabla_sh_{mn}.\]
\end{lemma}
\begin{proof}
Note that
\begin{align*}
2\frac{(\Theta-n\lambda H)^2}{H}=&2\frac{(\Delta H)^2}{H}+4|A|^2\Delta H-4\frac{\Delta H}{H}b^{ij}\nabla_iH\nabla_jH\\
&+2|A|^4H-4|A|^2b^{ij}\nabla_iH\nabla_jH\\
&+\frac{2}{H}\left(b^{ij}\nabla_iH\nabla_jH\right)^2
\end{align*}
and
\begin{align*}
2\{g^{mq}b^{np}&-\frac{g^{mn}g^{pq}}{H}\}\eta_{mn}\eta_{pq}\\
=&2g^{mq}b^{np}\nabla_m\nabla_nH\nabla_p\nabla_qH-2\frac{(\Delta H)^2}{H}\\
&+4Hh^{ij}\nabla_i\nabla_jH-4|A|^2\Delta H\\
&-4g^{mq}b^{np}b^{rs}\nabla_m\nabla_nH\nabla_rH\nabla_sh_{pq}+4\frac{\Delta H}{H}b^{ij}\nabla_iH\nabla_jH\\
&+2H^2C-2|A|^4H\\
&-2Hb^{ij}\nabla_iH\nabla_j|A|^2+4|A|^2b^{ij}\nabla_iH\nabla_jH\\
&+2\{b^{ir}b^{ks}b^{jl}\}\nabla_iH\nabla_jHg^{pq}\nabla_ph_{rs}\nabla_qh_{kl}-\frac{2}{H}\left(b^{ij}\nabla_iH\nabla_jH\right)^2.
\end{align*}
Now the claim follows from adding up these last two identities and considering Lemmas \ref{lem: lem1} and \ref{lem: lem2}.
\end{proof}
\section{Harnack estimate and Backwards Convergence}
\subsection{Harnack estimate}
If $M_0$ is not an equator, the strong parabolic maximum principle and the evolution equation of $h_i^j$ in Lemma \ref{lem: lem1} imply that for any $t>0$, $M_t$ is strictly convex. For strictly convex hypersurfaces, $$A(V,V)+2D_VH$$ is minimized by $V=(V^1,\ldots,V^i=-b^{ij}\nabla_iH,\ldots,V^n)$ and thus to prove the main theorem, it suffices to verify that for all $t>0$
\[\partial_tH-b^{ij}\nabla_iH\nabla_jH-n\lambda H +\frac{H}{2t}\geq 0.\]
Fix $0<\varepsilon<T.$ We will apply the maximum principle to
$$Q:=\frac{\Theta-n\lambda H}{H}=\frac{\Delta H+|A|^2H-b^{ij}\nabla_iH\nabla_jH}{H}$$
on the time interval $[\varepsilon, T).$ Using Lemmas \ref{lem: lem3}, \ref{lem: lema6}, and the inverse-concavity of the mean curvature, we calculate
\begin{align*}
\partial_tQ=\partial_t \frac{\Theta}{H}\geq&\Delta Q+2\langle \nabla Q,\frac{\nabla H}{H}\rangle+ 2Q^2+2\lambda H^2\\
&+2\frac{\{g^{mq}b^{np}-\frac{g^{mn}g^{pq}}{H}\}\eta_{mn}\eta_{pq}}{H}\\
\geq& \Delta Q+2\langle \nabla Q,\frac{\nabla H}{H}\rangle+ 2Q^2.
\end{align*}
An ODE comparison with $q(t)=-\frac{1}{2(t-\varepsilon)}$ which satisfies $\frac{d}{dt}q(t)=2q^2(t)$ and $\lim\limits_{t\to \varepsilon^+}q(t)=-\infty$ shows that $Q(\cdot,t)\geq q(t)$ for any $t>\varepsilon.$ Allowing $\varepsilon\to 0$ completes the proof of the Harnack estimate.
\subsection{Backwards convergence}
\begin{lemma}\label{cor: backward limit of principal curvatures}
Any weakly convex ancient solution of the mean curvature flow satisfies
\[|A|\leq c_0\exp(nt)\]
for $t\le 0$
for some $c_0$ depending only on $M_0.$
\end{lemma}
\begin{proof}
By the strong  parabolic maximum principle, any weakly convex ancient solution of the mean curvature flow must become strictly convex, unless it is a non-moving equator. By our Harnack inequality for strictly convex, ancient solutions we have
\[\frac{\partial_tH-b^{ij}\nabla_iH\nabla_jH-n\lambda H}{H}\geq 0.\]
Since $\lambda =1$, we get
\[\partial_t\log H\geq n.\]
Integrating both sides of this inequality against $dt$ on the time interval $[t,0]$ gives
\[H(\cdot,t)\leq H(\cdot,0)\exp(nt).\]
\end{proof}
\begin{lemma}
Any weakly convex ancient solution of the mean curvature flow satisfies
\[|\nabla^mA|^2\leq c_m\exp(2nt)\]
for $t\le 0,$ where $c_m$ depends only on $M_0$ and $m.$
\end{lemma}
\begin{proof}
The proof follows by induction on $m.$ Using the fourth evolution equation in Lemma \ref{lem: lem3}, $\partial_t\Gamma_{ij}^k=A\ast\nabla A$ and that the commutator $[\nabla^k,\Delta]T $ is given by
\[[\nabla^k,\Delta]T =\sum\limits_{j=0}^k\nabla^jRm\ast\nabla^{k-j}T\] for any tensor $T$, we can compute the following evolution equations:
\begin{enumerate}
  \item $\partial_t |A|^2=\Delta |A|^2-2|\nabla A|^2+2|A|^4+2\lambda(2H^2-n|A|^2)$
  \item \begin{align*}
  \partial_t |\nabla^mA|^2=&\Delta |\nabla^mA|^2-2 |\nabla^{m+1}A|^2\\
  &+\sum\limits_{i+j+k=m}\nabla^iA\ast\nabla^jA\ast\nabla^kA\ast\nabla^mA\\
  &+\lambda\nabla^mA\ast\nabla^mA.
         \end{align*}
\end{enumerate}
On the other hand, by Lemma \ref{cor: backward limit of principal curvatures} we get
\[\partial_t |A|^2\leq\Delta |A|^2-2|\nabla A|^2+c_1\exp(2nt),\]
and
\[\partial_t |\nabla A|^2\leq\Delta |\nabla A|^2+c_2|\nabla A|^2+c_3\exp(2nt).\]
Here, $c_1,c_2,c_3\geq0$ are independent of $t.$ Therefore,
\begin{align*}
\partial_t ((t-s)|\nabla A|^2+b|A|^2)\leq&\Delta((t-s)|\nabla A|^2+b|A|^2)\\
&+(1-2b+(t-s)c_2)|\nabla A|^2+(c_1b+c_3)\exp(2nt).
\end{align*}
We want to apply the maximum principle on the time interval $[s,s+1].$
We choose $b$ large enough, independent of $t$, so that the coefficient of $|\nabla A|^2$ becomes negative. Thus the maximum principle implies
\[(t-s)|\nabla A|^2\leq c_4\exp(2n t)(1-\exp(2n(s-t) ))+bc_0^2 \exp(2ns).\]
In particular for $t=s+1$ we have
\[|\nabla A|^2(\cdot, t)\leq c_5\exp(2n t).\]
This verifies the bound on $|\nabla A|.$ Higher derivative estimates follow by induction and using the auxiliary function
$(t-s)|\nabla^mA|^2+b_{m-1}|\nabla^{m-1}A|^2:$
\[\partial_t |\nabla^{m-1} A|^2\leq\Delta |\nabla^{m-1} A|^2-2 |\nabla^{m}A|^2+c_1\exp(2nt).\]
\[\partial_t |\nabla^m A|^2\leq\Delta |\nabla^m A|^2+c_2|\nabla^m A|^2+c_3\exp(2nt).\]
\end{proof}
Having established the higher derivative curvature bounds, convexity of $M_t$ implies that the backwards limit of $M_t$ is independent of subsequences. Therefore we have proved the following theorem.
\begin{proposition}\label{thm8}
Let $M_t$ be an ancient, embedded, weakly convex solution of the mean curvature flow. Then $M_t$ converges exponentially fast in $C^{\infty}$ to a unique equator $M_{-\infty}$ as $t\to-\infty.$
\end{proposition}
\section{Classification of Ancient Solutions}
In this section, we use Proposition \ref{thm8} and the Aleksandrov reflection \cite{Br-Lou} to classify convex, embedded ancient solutions of the mean curvature flow on \(\sphere^{n+1}\).

We will work relative to the limiting equator obtained in Proposition \ref{thm8}. Let \(E\) be an equator that bounds the \emph{open} hemispheres \(H^{\pm}\) with centers \(\pm \basepoint\), and let \(\vertvec = \overrightarrow{\origin\basepoint}\) be the unit vector in \(\RR^{n+2}\) that points from the origin \(\origin\) to \(\basepoint\). Let \(\radialdistance(x) = d_{\sphere^{n+1}} (\basepoint, x)\) denote the spherical distance from \(\basepoint\) to \(x \in \sphere^{n+1}\). The radial projection onto \(E\) is the map \(x \in \sphere^{n+1} \mapsto \radialprojection(x) \in E\), where \(\radialprojection\) is the nearest point on \(E\) to \(x\). If \(x \ne \pm \basepoint\), then \(\radialprojection(x)\) is a single point. If \(x = \pm \basepoint\), then \(\radialprojection(x) = E\). In any event, given \(y \in \radialprojection(x)\), there is a unique length minimizing geodesic joining \(x\) to \(y\) and this geodesic must pass through \(\pm \basepoint\).

It is convenient to make use of the ambient \(\RR^{n+2}\) and define the height function \(\height(x) = \ip{x}{\vertvec}\). The radial distance is related to the height function via
\[
\height(x) = \cos(\radialdistance(x))
\]
which is monotonically decreasing in \(\radialdistance\).

For the Aleksandrov reflection, let \(\reflectionvector \in \RR^{n+2}\) be any unit vector that \(\ip{\reflectionvector}{\vertvec} < 0\). Let \(\reflectionplane = \reflectionvector^{\perp}\) be the hyperplane through the origin with the normal vector \(\reflectionvector\). Let \(\reflectionhalfspace^{\pm} = \{\pm \ip{x}{\reflectionvector} > 0\}\) denote the halfspaces with the boundary \(\reflectionplane\). For any subset \(S \subset \sphere^{n+1}\), write \(\reflectionset{S}^{\pm} = S \intersect \reflectionhalfspace^{\pm}\). Lastly, let \(\delta > 0\) denote the angle \(\reflectionvector\) makes with \(E\); that is, \(\delta = \arcsin\ip{\reflectionvector}{-\vertvec}\).
\begin{definition}
The Aleksandrov reflection across \(\reflectionplane\) is the map defined by
\[
\reflectionmap: x \in \RR^{n+1} \mapsto x - 2\ip{x}{\reflectionvector} \reflectionvector.
\]
\end{definition}
This map is an (orientation reversing) isometry of \(\RR^{n+2}\) fixing \(\reflectionplane\) and in particular fixing the origin. Therefore, it induces an isometry of \(\sphere^{n+1}\). For \(x \in E\), we have \(\ip{x}{\vertvec} = 0\) and
\[
\height(\reflectionmap(x)) = \ip{\vertvec}{x - 2 \ip{x}{\reflectionvector} \reflectionvector} = 2\ip{x}{\reflectionvector}\sin\delta .
\]
In the case \(x \in \reflectionset{E}^+\), \(\height(\reflectionmap(x)) > 0\), and in the case \(x \in E \intersect \reflectionplane\), \(\height(\reflectionmap(x)) = 0\).


Now we outline the proof of the main theorem and highlight the main steps. We refer the reader to \cite{BIS2,Br-Lou} for details and generalizations to a large class of \emph{fully non-linear} curvature flows.
\begin{theorem}
Let \(M_t\) be a convex, embedded ancient solution of the mean curvature flow on \(\sphere^{n+1}\). Then \(M_t\) is a family of shrinking geodesic spheres emanating from an equator at \(t=-\infty\).
\end{theorem}
\begin{proof}
Let \(E = M_{-\infty}=\lim_{t\to-\infty} M_t\) be the limiting equator. Since \(M_t\) smoothly converges to \(E\) as \(t\to-\infty\), we may write \(M_t\) as the graph of a smooth positive function over \(E\) in the geodesic polar coordinates, \(M_t = \{(f_t(\sigma), \sigma) \in (0,\pi) \times E\)\}. We have \(f_t \to \pi/2\) smoothly and uniformly in \(C^{k}\), for any $k$, as \(t \to -\infty\). Moreover, for \(\delta \in (0,\pi/4)\), \(\reflectionmap(E) = \{(g(\sigma), \sigma)\)\} is a graph over \(E\). Note that \(\reflectionmap(E)\) is not a graph when \(\delta = \pi/4\) (it is an equator perpendicular to \(E\)), so let us fix a \(\delta_0 \in (0,\pi/4)\) to give us a little room. Therefore, for $S$ sufficiently small and independent of \(\delta \in (0, \delta_0)\),  both \(\reflectionset{(M_t)}^-\) and \(\reflectionmap(\reflectionset{(M_t)}^+)\) are non-empty and are graphs for all times in \(t\in(-\infty, S)\).

We will show that \(\reflectionmap(\reflectionset{(M_t)}^+) > \reflectionset{(M_t)}^-\) away from a strip. As noted above, \(\height(\reflectionmap(x)) > 0\) on \(\reflectionset{E}^+\) and \(\height(x) = \cos(\radialdistance(x))\). Thus, continuity implies that for any \(\varepsilon > 0\) there exists an \(\eta>0\) such that \(\radialdistance(\reflectionmap(x)) < \pi/2 - \varepsilon\) provided \(x \in E_{\eta} := \{x \in E: d(x, E \intersect \reflectionplane) > \eta\}\). Since \(M_t \to_{C^{\infty}} E\), we can choose \(T_{\delta} < 0\), such that \(d(M_t, E) < \varepsilon/2\) for all \(t < T_{\delta}\); that is, \(\radialdistance (x) > \pi/2 - \varepsilon/2\) for all \(x \in M_t\). Now for \(x \in \reflectionset{(M_t)}^+ \intersect \radialprojection^{-1} (E_{\eta})\), since \(\reflectionmap\) is an isometry, we have \(d(\reflectionmap(x), \reflectionmap(\radialprojection(x))) < \varepsilon/2\); therefore, \(\radialdistance(\reflectionmap(x)) < \pi/2 - \varepsilon/2\). Consequently, away from the strip \(E\backslash E_{\eta}\), we have \(\radialdistance(\reflectionmap(\reflectionset{(M_t)}^+)) < \pi/2 - \varepsilon/2\) and \(\radialdistance(\reflectionset{(M_t)}^-) > \pi/2 - \varepsilon/2\). That is,  \[\reflectionmap(\reflectionset{(M_t)}^+) > \reflectionset{(M_t)}^-\quad \mbox{away from the strip.}\]

By \cite[Proposition 6.4]{BIS2}, we also have \(\reflectionmap(\reflectionset{(M_t)}^+) \geq \reflectionset{(M_t)}^-\) on the strip \(E\backslash E_{\eta}\). Thus we find that
\begin{equation}
\label{eq:backwards_approximate_symmetry}
\reflectionmap(\reflectionset{(M_t)}^+) \geq \reflectionset{(M_t)}^-\quad\mbox{for all}~ t \in (-\infty, T_{\delta}),
\end{equation}
where \(T_{\delta}\) depends only on \(\delta\).

We define \(T_{\delta}\) so that \((-\infty, T_{\delta})\) is the largest interval on which the relation \eqref{eq:backwards_approximate_symmetry} holds and \(T := \inf_{\delta \in (0,\delta_0)} T_{\delta}\). We want to show that \(T > -\infty\), and hence that the relation \eqref{eq:backwards_approximate_symmetry} holds on the non-empty interval \((-\infty, T)\). To show $T>-\infty$, we apply the maximum principle (see \cite[Lemma 6.5]{BIS2}): Recall that both \(\reflectionset{(M_t)}^-\) and \(\reflectionmap(\reflectionset{(M_t)}^+)\) are non-empty and are graphs for all \(t\in(-\infty, S)\) with \(S>-\infty\) independent of \(\delta\). Since the relation \eqref{eq:backwards_approximate_symmetry} is true on \((-\infty, \min\{T_{\delta},S\})\), the maximum principle applied in the time interval $[\min\{T_{\delta},S\}/2,S)$ ensures that
\begin{equation*}
\label{eq:longtime_approximate_symmetry}
\reflectionmap(\reflectionset{(M_t)}^+) \geq \reflectionset{(M_t)}^-
\end{equation*}
for all \(t \in (-\infty, S)\) and any \(\delta \in (0,\delta_0)\). Hence $T\ge S>-\infty.$

To finish the proof, we use \cite[Proposition 5.3]{Br-Lou} to conclude that \(M_t\) is a geodesic sphere for all \(t \in (-\infty, T)\) and thus for all negative times by the uniqueness of solutions.
\end{proof}
\bibliographystyle{amsplain}

\begin{thebibliography}{10}
\bibitem{Ang} S. Angenent, ``Formal asymptotic expansions for symmetric ancient ovals in mean curvature flow." Networks and Heterogeneous Media \textbf{8}(2013): 1--8.
\bibitem{Andrews 94} B Andrews, ``Harnack inequalities for evolving hypersurfaces." Mathematische Zeitschrift \textbf{217}(1994): 179--197.
\bibitem{arr-sun 13} C. Arezzo and J. Sun, ``Conformal solitons to the mean curvature flow and minimal submanifolds." Mathematische Nachrichten \textbf{286}(2013): 772--790.
\bibitem{BIS2} Bryan, Paul and Ivaki, Mohammad N. and Scheuer, Julian, ``On the classification of ancient solutions to curvature flows on the sphere'' arXiv Preprint (2016), arXiv:1604.01694.
\bibitem{Br-Lou} P. Bryan and L. Janelle, ``Classification of convex ancient solutions to curve shortening flow on the sphere." Journal of Geometric Analysis, \textbf{26}(2016): 858--872.
\bibitem{Hamilton 95} R.S. Hamilton, ``Harnack estimate for the mean curvature flow." Journal of Differential Geometry \textbf{41}(1995): 215--226.
\bibitem{Chow 97} B. Chow, ``Geometric aspects of Aleksandrov reflection and gradient estimates for parabolic equations." Communications in Analysis and Geometry \textbf{5}(1997): 389--409.
\bibitem{Chow-Gul 01} B. Chow and R. Gulliver, ``Aleksandrov reflection and geometric evolution of hypersurfaces." Communications in Analysis and Geometry \textbf{9}(2001): 261--280.
\bibitem{Chow-Gul 96} B. Chow and R. Gulliver, ``Aleksandrov reflection and nonlinear evolution equations. I. The $n$-sphere and $n$-ball." Calculus of Variations and Partial Differential Equations \textbf{4}(1996): 249--264.
\bibitem{Da-Ham-Sesu 2010} P. Daskalopoulos, R. Hamilton and N. Sesum, ``Classification of compact ancient solutions to the curve shortening flow." Journal of Differential Geometry \textbf{84}(2010): 455--464.
\bibitem{Da-Ham-Sesu 2012} P. Daskalopoulos, R. Hamilton and N. Sesum, ``Classification of ancient compact solutions to the {R}icci flow on surfaces." Journal of Differential Geometry \textbf{91}(2012): 171--214.
\bibitem{Huisken 87} G. Huisken, ``Deforming hypersurfaces of the sphere by their mean curvature." Mathematische Zeitschrift \textbf{195}(1987): 205--219.
\bibitem{Hu-Sin 1999} G. Huisken and C. Sinestrari, ``Mean curvature flow singularities for mean convex surfaces." Calculus of Variations and Partial Differential Equations \textbf{8}(1999): 1--14.
\bibitem{Hu-Sin 2014} G. Huisken and C. Sinestrari, ``Convex ancient solutions of the mean curvature flow." Journal of Differential Geometry, \textbf{101}(2015): 267--287.
\bibitem{Has-Her} R. Haslhofer and O. Hershkovits, ``Ancient solutions of the mean curvature flow." Communications in Analysis and Geometry \textbf{24}(2016): 593--604.
\bibitem{hun-nor 12} N. Hungerb{\"u}hler and T. Mettler, ``Soliton solutions of the mean curvature flow and minimal hypersurfaces." Proceedings of the American Mathematical Society \textbf{140}(2012): 2117--2126.
\bibitem{smo 97} K. Smoczyk, ``Harnack inequalities for curvature flows depending on mean curvature." New York Journal of Mathematics \textbf{3}(1997): 103--118.
\bibitem{smo 01} K. Smoczyk, ``A relation between mean curvature flow solitons and minimal submanifolds." Mathematische Nachrichten \textbf{229}(2001): 175--186.
\bibitem{Wang} X-J.Wang, ``Convex solutions to the mean curvature flow." Annals of Mathematics. Second Series \textbf{173}(2011): 1185--1239.
\bibitem{Whi} B. White, ``The nature of singularities in mean curvature flow of mean convex sets." Journal of the American Mathematical Society \textbf{16}(2003): 123--138.
\bibitem{Urbas 91} J.I.E. Urbas, ``An expansion of convex hypersurfaces." Journal of Differential Geometry \textbf{33}(1991): 91--125.
\end{thebibliography}

\end{document}